\documentclass[a4paper,11pt]{amsart}
\usepackage{graphicx}
\usepackage{amssymb}
\usepackage{epstopdf}
\usepackage{nicefrac}
\usepackage{mathrsfs}
\usepackage{amsfonts}
\usepackage{amsthm}
\usepackage{enumerate}
\usepackage{graphicx}
\usepackage{comment}

\newtheorem{thm}{Theorem}[section]

\newtheorem{lemma}[thm]{Lemma}

\newtheorem{prop}[thm]{Proposition}

\theoremstyle{remark}
\newtheorem{rem}[thm]{Remark}
\theoremstyle{theorem}

\theoremstyle{definition}
\newtheorem{defn}[thm]{Definition}

\newcommand{\exo}{\Sigma}

\newcommand{\mR}{{\mathbb R}}

\newcommand{\mZ}{{\mathbb Z}}

\newcommand{\cD}{{\mathcal D}}

\newcommand{\cL}{{\mathcal L}}
\newcommand{\cM}{{\mathcal M}}

\newcommand{\Met}{\mathcal{MET}}
\newcommand{\cT}{{\mathcal T}}

\newcommand{\lra}{\longrightarrow}
\newcommand{\bgd}{\begin{displaymath}}
\newcommand{\edd}{\end{displaymath}}

\begin{document}

\title[Teichm\"uller space of Complex Hyperbolic Manifolds]{Teichm\"uller space of negatively curved metrics on Complex Hyperbolic Manifolds is not contractible}

%\thanks{2000 {\it Mathematics Subject Classification}. Primary 57R30, 37C55, 37B45; Secondary 53C12 }
\author{F. T. Farrell}
\address{Yau Mathematical Sciences Center and Department of Mathematics, Tsinghua University, Beijing, China}
\email{farrell@math.tsinghua.edu.cn}
%\thanks{We thank the journal Science China for accepting this article for publication.}
 
\author{G. Sorcar}
\address{Department of Mathematics, Ohio State University, Columbus, OH 43210}
\email{sorcar.1@osu.edu}
%\thanks{date: January 12, 2014}

\date{}
 

\maketitle

\begin{abstract}
In this paper we prove that for all $n=4k-2$, $k\ge2$ there exists a closed smooth complex hyperbolic manifold $M$ with real dimension $n$
having non-trivial $\pi_1(\cT^{<0}(M))$. $\cT^{<0}(M)$
denotes the Teichm\"uller space of all negatively curved Riemannian metrics on $M$, which is the topological quotient of the space of all 
negatively curved metrics modulo the space of self-diffeomorphisms of $M$ that are homotopic to the identity.
\end{abstract}

\tableofcontents

\section{Introduction}\label{sec-intro}

This paper builds on arguments from the paper \cite{SOR} that proves a similar result for Gromov-Thurston manifolds that support negatively curved
Riemannian metrics but are not hyperbolic. Let us recall some terminology from that paper:
\begin{itemize}
\item $\Met(M)$ denotes the space of all Riemannian metrics on $M$ with smooth topology. [Note that $\Met(M)$ is contractible. 
Any two metrics can be joined by a line segment in the space of metrics, as the convex combination of two metrics is also a metric.]
\item Diff$(M)$ is the group of all smooth self-diffeomorphisms of $M$. $\text{Diff}(M)$ acts on $\Met(M)$ by pushing forward metrics, 
that is, for any $f$ in Diff$(M)$ and any metric $g$ in $\Met(M)$, $f_*g$ is the metric such that $f:(M,g)\rightarrow(M,f_*g)$ is an isometry.
\item $\text{Diff}_0(M)$ stands for the subgroup of Diff$(M)$ consisting of all smooth self diffeomorphisms of the manifold $M$ 
that are homotopic to the identity. 
\item $\cD_0(M)$ is the group $\mR^+\times$ Diff$_0(M)$. 

\item $\cD_0(M)$ acts on $\Met(M)$ by scaling and pushing forward metrics, that is, when $(\lambda,f) \in \cD_0(M)$ and $g \in$ $\Met(M)$, 
$(\lambda,f)g=\lambda(f)_*g$. The Teichm\"uller space of all metrics on $M$ is defined to be $\cT(M):=\Met(M)/\cD_0(M)$.
\item Similarly, the Teichm\"uller space of negatively curved metrics on $M$ is defined as $\cT^{<0}(M):=\Met^{<0}(M)/\cD_0(M)$.
\end{itemize}

In this paper we prove:
\begin{thm}\label{mainthm}
For every positive integer $n=4k-2$ where $k$ is an integer more than $1$, there is a complex hyperbolic manifold $M^n$
such that $\pi_1(\cT^{<0}(M^n)) \ne 0$. Therefore $\cT^{<0}(M)$ is not contractible. 
\end{thm}

\begin{rem} At the end of section $4$ of this paper we construct $M$ and we also describe a very specific negatively curved metric
$g_s$ on $M$ with certain geometric properties.
 The basepoint in $\cT^{<0}(M)$ for $\pi_1(\cT^{<0}(M))$ is the equivalence class of this metric $g_s$ on $M$. This equivalence class is 
 well defined because of Mostow's Strong Rigidity Theorem \cite{MOS}.
\end{rem}

{\bf {Idea of proof:}} Suppose $M$ is our complex hyperbolic manifold. Consider the sequence:\\
\bgd
 \cD_0(M) \lra \Met^{<0}(M) \lra \frac{\Met^{<0}(M)}{\cD_0(M)}=:\cT^{<0}(M).
\edd

By the work of Borel, Conner and Raymond \cite{BCR} one gets that $\cD_0(M)$ acts freely on $\Met(M)$ and more details on this can be found in page 51 of 
\cite{FO1}. Then by using Ebin's Slice Theorem one can deduce that the above sequence is a fibration.

Hence from the above fibration we get a long exact sequence in homotopy, part of which we use and is shown below:
\bgd
\pi_1(\Met^{<0}(M)) \lra \pi_1(\cT^{<0}(M)) \lra \pi_0(\cD_0(M)) \lra \pi_0(\Met^{<0}(M)).
\edd

The basepoint for $\Met^{<0}(M)$ is the negatively curved metric $g_s$ on $M$ that we describe in section $4$ and
the basepoint for $\cD_0(M)$ is the identity map from $M$ to itself.

We want to come up with a $y \ne 0$ in $\pi_1(\cT^{<0}M)$. To this end, we shall first construct in section $2$ an 
$f \in \cD_0(M)$ by changing the identity map on $M$ in a strategically placed geodesic annulus in $M$. This annulus is located 
using properties of the metric $g_s$ and the change in the identity map is brought about by using exotic spheres in a dimension 
higher than that of $M$. 

Then in section $3$ we show that
$[f] \ne 0$ in $\pi_0(\cD_0(M))$ using properties of the exotic sphere we used to construct $f$. In section $4$ we proceed to prove that 
$[f]$ maps to zero in $\pi_0(\Met^{<0}(M))$ using the strategic placement of the annulus in $M$ and properties of $g_s$. 
Since $[f]$ maps to zero, it can be pulled back in $\pi_1(\cT^{<0}(M))$, and since the pullback of a non-zero element cannot be zero,
this pullback serves as the $y\ne0$ in $\pi_1(\cT^{<0}(M))$ we are looking for. 

\begin{comment}

In \cite{FO1} Farrell and Ontaneda also follow a similar program for their proof, but the part of constructing $[f] \ne 0$ in $\pi_0(\cD_0(M))$ uses a result 
by Sullivan that says,any hyperbolic manifold has a finite sheeted cover that embeds with trivial normal bundle, which we cannot use here since 
the manifolds we are dealing with are not hyperbolic. 

Instead we use surgery with exotic spheres and information about exotic spheres that bound parallelizable manifolds from
the works of Kervaire and Milnor \cite{MIL-KER}, in section 3.2 to accomplish this. This method can also be applied to the hyperbolic manifolds being studied
in \cite{FO1} by Farrell and Ontaneda, because all that we use about the manifolds is the fact that they support metrics of negative curvature and that 
they have trivial total Pontryagin class.

\end{comment}

The work involved in showing that the $[f] \ne 0$ in $\pi_0(\cD_0(M))$ we construct in section 3.2 maps to zero in $\pi_0(\Met^{<0}(M))$ is basically 
exactly the same as in \cite{FO1}, and is done by tapering constructions of metrics. This has been elaborated in section 3.3.

Another thing worth noting is that eventually we shall only construct this $[f]\in \pi_0(\text{Diff}_0(M))$ and 
follow through with the above idea. This is fine because $\text{Diff}_0(M)$ is a deformation retract of $\cD_0(M)$.

\section{Construction of the diffeomorphism}\label{sec:construct}
An exotic $n$-sphere $\Sigma^n$ is an $n$-dimensional smooth manifold that is homeomorphic to $S^n$ but not diffeomorphic to $S^n$. $\Sigma^n$
is a twisted double of two copies of $D^n$ joined by an orientation preserving diffeomorphism of the boundary of $D^n$ denoted by $\partial D^n =
S^{n-1}$. Let this orientation preserving diffeomorphism be $\gamma_1:S^{n-1} \rightarrow S^{n-1}$. 

Let $P_1$ and $P_2$ be two antipodal points in $S^{n-1}$ and let $N(P_1)$ and $N(P_2)$ be some neighborhood of $P_1$ and $P_2$ respectively in $S^{n-1}$
with radius of the neighborhoods small enough so that $N(P_1)\cap N(P_2)=\varnothing$ and consequently $\gamma_1(N(P_1))\cap \gamma_1(N(P_2))=\varnothing$.

One can also show that this diffeomorphism $\gamma_1$ is smoothly isotopic to a 
diffeomorphism $\gamma_2:S^{n-1} \rightarrow S^{n-1}$  such that $\gamma_2$ is the identity restricted 
to the neighborhoods $N(P_1)$ and $N(P_2)$ and is homotopic to the identity map on $S^{n-1}$ relative to $N(P_1) \cup N(P_2)$.

Now, this diffeomorphism $\gamma_2$ restricted to $S^{n-2}\times [1,2]$ can be shown smoothly pseudo-isotopic to a self-diffeomorphism 
$h:S^{n-2}\times [1,2]\rightarrow S^{n-2}\times [1,2]$ such that, $h$ is level preserving, in other words, if we fix any 
$t \in [1, 2]$ then for all $x \in S^{n-2}$, $h(x,t)=(y,t)$ for some $y \in S^{n-2}$. We shall denote this $y$ as $h_t(x)$.

This construction has been discussed in more details with the references and arguments needed to carry it out in \cite{SOR}.

The reason for taking smaller neighborhoods inside the original neighborhoods of the two points $P_1$ and $P_2$ is because we want the
self-diffeomorphism $h$ of $S^{n-2}\times [1,2]$ to be the identity near $1$ and $2$, which is desirable due to technical reasons. 

If we select an exotic sphere $\Sigma$ of dimension $4k-1$, we can get a similar diffeomorphism $h:S^{4k-3}\times [1,2]\rightarrow S^{4k-3}\times [1,2]$ 
by the above process. With the help of this diffeomorphism $h$ we shall construct a diffeomorphism 
$f:M\rightarrow M$ such that $[f] \in \pi_0(\text{Diff}_0(M))$.

%say something about embedding the annulus around a point in M with high injectivity radius.
We choose a real number $\alpha$ and a point $p \in M$ such that the injectivity radius of the metric $g_s$ at $p \in M$ is greater than $2\alpha$.
We also want to choose this $p$ in such a way that the metric $g_s$ restricted to a closed geodesic ball of radius $3\alpha$ around $p$ has a constant 
sectional curvature  of $-1$, i.e. is the hyperbolic metric. The justification of being able to choose such a point $p$ is due to the work in \cite{FJ3} and
will be discussed briefly later on.
This geometric piece of information will not be used until in section $4$ of this paper.

We are now in a position to define our self-diffeomorphism $f$ on $M$. To interpret the formula for $f$ that we give below, one needs to identify a
a closed geodesic ball of radius $2\alpha$ centered at $p$ minus the point $p$ with $S^{4k-3} \times (0,2\alpha]$.
The lines $t \mapsto (x,t)$ in $S^{4k-3} \times (0,2\alpha]$ are identified with the unit speed geodesics emanating from $p$ in the geodesic ball. 

We will now define $f \in \text{Diff}(M)$ as follows: \\\\
\centerline {$\displaystyle f(q)  = \begin{cases} 
      q & \text{if } q \notin (S^{4k-3}\times[\alpha,2\alpha])\subset M \\
      (h_{t/\alpha}(x),t) & \text{if } q=(x,t) \in (S^{4k-3}\times[\alpha,2\alpha]) \subset M \\
     \end{cases}$}\\\\
where $h(x,s)=(h_s(x),s)$.\newline

The diffeomorphism $h$ is homotopic to the identity map on $S^{4k-3}\times [1,2]$ relative to the boundary, because
$h_2$ is homotopic to the identity map of $S^{4k-2}$ relative to the neighborhoods $N(P_1)$ and $N(P_2)$. Therefore $f$ constructed as above is
homotopic to the identity on $M$. Hence $f \in \text{Diff}_0(M)$. 

\section{Showing $f$ is not smoothly isotopic to the identity} We will now show that there exists an exotic sphere $\Sigma$
in each dimension $4k-1 (k\ge 2)$ such that there is no path in Diff$(M)$ connecting $f$ (built using this $\Sigma$ as explained in the previous section)
and the identity map on $M$. In other words we shall show that $f$ is not smoothly isotopic to the identity on $M$. This will be achieved by assuming
that there is such a smooth isotopy and arriving at a contradiction.\\\\
This section also uses most arguments from \cite{SOR} but since the Pontryagin numbers of a complex manifold need not all vanish, 
we have to change some arguments for this paper.\\

Before we reach the point where we assume the existence of the smooth isotopy between $f$ and the identity on $M$ let us provide ourselves with some basic set up.

A definition that will be used throughout the section:

\begin{defn}
 Given a smooth manifold $W$ with collared boundary $\partial_1 W \sqcup \partial_2 W$ and a diffeomorphism $F:\partial_1 W \rightarrow \partial_2 W$, we obtain a
 smooth manifold
 without boundary, $W_F := W/x\sim F(x)$ where $x \in \partial_1 W$.
\end{defn}
Note: If we have a homeomorphism or a diffeomorphism $f:N\rightarrow N$ then $(N\times I)_f$ is just the usual mapping torus. Also, a ``collared boundary'' is
a boundary along with a chosen collar neighborhood of the boundary in the manifold.

\vspace{5mm}
Two lemmas that we shall use are as follows:
\begin{lemma}\label{lem:lem1}
 If $\alpha$ and $\beta$ are two homotopic self-homeomorphisms of a non positively curved closed manifold $N$ then $\alpha$ and $\beta$ are topologically 
 pseudo-isotopic provided the dimension of $N$ is greater or equal to $4$.
\end{lemma}

\begin{lemma}\label{lem:lem2}
 In the context of Definition $3.1$ if 
 $\alpha$ and $\beta$ are two topologically pseudo isotopic homeomorphisms from $\partial_2 W$ to $\partial_1 W$, then $W_\alpha$ is homeomorphic
 to $W_\beta$.
\end{lemma}

The second lemma is not difficult to prove. The first lemma is a result by Farrell and Jones \cite{FJ1}.

In all that follows now, $\Sigma$ stands for the exotic sphere of dimension $4k-1$ used in section \ref{sec:construct}, and $id_N$ will denote
the identity map on a manifold $N$.

The following proposition is a fact that is true because of the way our $f$ has been constructed. 
We are not yet assuming $f$ to be smoothly isotopic to $id_M$.

\begin{prop}
There is a diffeomorphism $\mathcal{F}_1:(M\times S^1)\# \exo\rightarrow(M\times I)_f$.
\end{prop}

\begin{proof}
First let us observe that,
$$M\times S^1 \# \exo = (M \times S^1 \setminus D^{4k-1}) \bigsqcup_{id} (\exo \setminus D^{4k-1})$$
Now, if we define $\Hat{h}_2=h_2|_{S^{4k-2}\setminus N(P_1)}$, then
$$\exo\setminus D^{4k-1}=D^{4k-1}\bigsqcup_{\hat{h}_2} D^{4k-1}$$
This has been illustrated in Fig 1.
In the left hand side picture we are showing how $\exo$ is formed by identifying the boundary of two copies of $D^{4k-1}$ by $h_2$ (we are
thinking of $D^{4k-1}$ as a solid ball with boundary$S^{4k-2}$).
On the right hand side picture we are showing how partially identifying the boundary by the restricted map $\hat{h}_2$ we obtain
$\exo\setminus D^{4k-1}$. The shadings are there to help understand the partial identification better and relate it further to the rest of the proof.

$$\includegraphics[width=13cm,height=6.5cm,scale=3]{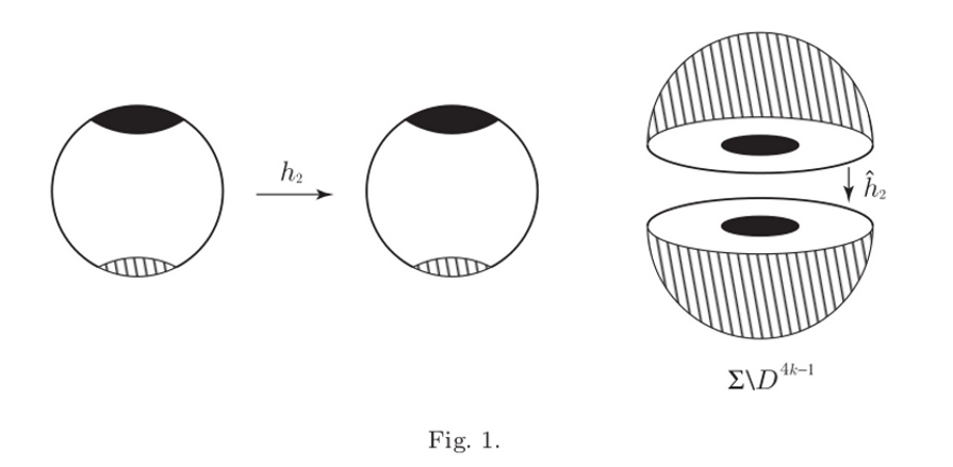}$$

So we now have,
$$M\times S^1 \# \exo = (M \times S^1 \setminus D^{4k-1}) \bigsqcup_{id} (D^{4k-1}\bigsqcup_{\hat{h}_2} D^{4k-1}) $$
and the right hand side of this equation is $(M\times I)_f$, which is being illustrated in Fig. 2:

$$\includegraphics[width=13cm,height=9cm]{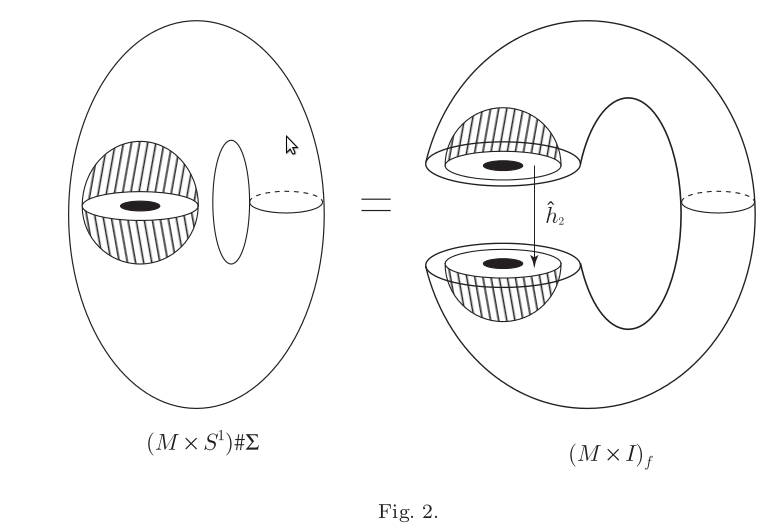}$$

\end{proof}

Finally we are going to assume that if possible $f$ be smoothly isotopic to $id_M$ and make the following observation:

\begin{prop}\label{prop:pi}
 If $f$ is smoothly isotopic to $id_M$ then, there is a diffeomorphism $\mathcal{F}_2:M\times S^1\rightarrow (M\times I)_f$.
 Also, the induced map ${\mathcal{F}_2}_*:\pi_1(M\times S^1)\rightarrow \pi_1((M\times I)_f )$ restricted to 
 $\pi_1(M) \subset \pi_1((M\times I)_f)$ is the identity map onto $\pi_1(M) \subset \pi_1(M\times S^1)$.
\end{prop}

\begin{proof}
Since $f$ is smoothly isotopic to $id_M$ we have a diffeomorphism $F:M\times I\rightarrow M\times I$ such that $F|_{M\times \{0\}}=id_M$ and
$F|_{M\times \{1\}}=f$. Now this $F$ induces a diffeomorphism in the quotient space $\mathcal{F}_2:(M\times I)_{id}\rightarrow (M\times I)_f$,
as well-definedness follows from the definition of $F$.
\\\\
 Now the basepoint can be chosen to be in $M\times\{0\}$ in both spaces $M\times S^1$ and $(M\times I)_f$
 and we also identify $M$ with $M\times\{0\}$.
 Since $\mathcal{F}_2$ preserves the level $M\times \{0\}$ we have
 ${\mathcal{F}_2}_*:\pi_1((M\times I)_f)\rightarrow \pi_1(M\times S^1)$ restricted to 
  $\pi_1(M) \subset \pi_1(M\times S^1)$ is the identity map onto $\pi_1(M) \subset \pi_1((M\times I)_f)$.
\end{proof}

Composition of the two diffeomorphisms in the above two propositions yield a diffeomorphism $\mathcal{F}=\mathcal{F}_2^{-1}\circ\mathcal{F}_1$: 
\begin{displaymath}
 \mathcal{F}:(M\times S^1)\# \exo\rightarrow (M\times S^1)
\end{displaymath}

We now discuss one property of this diffeomorphism $\mathcal{F}$, as stated in the proposition
that follows. This property will be useful towards reaching a contradiction.
\begin{prop}\label{prop:pseu}
 $\mathcal{F}$ is topologically pseudo-isotopic to the identity map.
\end{prop}

\begin{rem}
 The term ``identity map'' has been abused in the above proposition because the source and target space of the map $\mathcal{F}$ 
 are not the same smooth manifolds. Hence saying $\mathcal{F}$ is not topological pseudo-isotopic to the identity does not really make
 sense. To be correct, we should note that there is a homeomorphism from $(M \times S^1)\# S^{4k-1}$ to $(M\times S^1)\# \exo$ and we
 call this homeomorphism $G_1$, also note that there is a homeomorphism from $M\times S^1$ to $(M \times S^1)\# S^{4k-1}$ and call this 
 homeomorphism $G_2$. Then the correct restatement of the above proposition is that $G_2^{-1} \circ F \circ G_1$ is topologically pseudo-isotopic to the 
 identity on $(M \times S^1)\# S^{4k-1}$. We thank the referee for pointing out this detail.\\
 
 In an attempt to make the reading less notation heavy, we keep the proposition free from the above mentioned rigor and consider both
 $(M\times S^1)\# \exo$ and $M\times S^1$ to stand for their underlying topological spaces and since they are homeomorphic, we can think of
 them as being the same topological space, hence the use of the term ``identity map''.
 Also when we talk of $\mathcal{F}$ we think of it is a homeomorphism forgetting the fact that 
 it is also a diffeomorphism. The same understanding has been adopted for the proof that follows.
\end{rem}

\begin{proof}
 We shall prove $\mathcal{F}$ is homotopic to the identity $id$. Because then by Lemma 3.2 we are done.
 %Also in this proof we shall assume $(M\times S^1)\# \exo$ and $M\times S^1$ are the same as everything is being done at the level of topology here.
 
 Let $p:M\times S^1\rightarrow M$ be defined as $p(m,x)=m$ and $q:M\times S^1\rightarrow S^1$ be defined as $q(m,x)=x$.
 Then we would like to make two claims as follows:\\\\
 Claim 1: $p$ is homotopic to $p\circ \mathcal{F}$.\\
 Claim 2: $q$ is homotopic to $q\circ \mathcal{F}$.\\
 
 Let the homotopies in the claims 1 and 2 be $k^1_t$ and $k^2_t$ respectively with $k^1_0=p, k^1_1=p\circ \mathcal{F}$ and $k^2_0=q, k^2_1=q\circ \mathcal{F}$.
 Then the required homotopy between $\mathcal{F}$ and $id$ is, 
 $$K_t:M\times S^1 \rightarrow M\times S^1, K_t(m,x)=(k^1_t(m,x),k^2_t(m,x)).$$
 
 Let us now see why Claim 1 is true. Denote by $\mathscr{H}(M)$ the space of all self homotopy equivalences of $M$. 
 Define $\alpha: S^1 \rightarrow \mathscr{H}(M)$ by $\alpha(x)(m)=p\circ \mathcal{F}(m,x)$.
 It is a fact that since $M$ is aspherical and center$(\pi_1(M))=1$ we have $\pi_1(\mathscr{H}(M))=1$.
 Hence the loop $\alpha$ is null-homotopic, giving a homotopy $g_t:S^1\rightarrow \mathscr{H}$ such that $g_0(x)(m)=p\circ F(m,x)$ and $g_1(x)(m)=m$.
 From this homotopy we derive the homotopy we need, defined as $G_t:M\times S^1\rightarrow M$ such that $G_t(m,x)=g_t(x)(m)$.\\
 
 Now let us see why Claim 2 is true. Define the map $\bar{q}:(M\times I)_f\rightarrow S^1$.
 To prove the claim we separately prove that $q$ is homotopic to $\bar{q}\circ \mathcal{F}_1$ and $q$ is homotopic to $\bar{q}\circ \mathcal{F}_2$.
 Since $S^1$ is a $K(\mathbb{Z},1)$ we can establish this by showing 
 $\bar{q}_\#=q_\#\circ {\mathcal{F}_1}_\#$ and $\bar{q}_\#=q_\#\circ {\mathcal{F}_2}_\#$, where the subscript $\#$ on any map denotes its corresponding
 induced map at the level of fundamental groups.
 These can be done by choosing the basepoints carefully and for the second equality we shall need the condition on homotopy we stated in Proposition 3.5.
\end{proof}

Let $T$ denote $M\times S^1$ from now onwards.

Assuming this diffeomorphism $\mathcal{F}:T\# \exo\rightarrow T$ exists we want to arrive at a contradiction. 

Using this hypothetical diffeomorphism $\mathcal{F}$ we will make a hypothetical smooth manifold $\mathscr{N}$ of dimension real dimension $4k$.

Let the exotic sphere $\exo$ we selected bound $W$, a $4k$-dimensional parallelizable manifold. In Kervaire and Milnor's paper \cite{MIL-KER}
we learn that such exotic spheres that bound parallelizable manifolds do exist in the dimensions that we are working with. We shall use information
on these special exotic spheres as and when we need from \cite{MIL-KER}.

Let us now denote the boundary connect sum of $T \times [0,1]$ and $W$ by $T\times [0,1] \hspace{1mm} \#_b  \hspace{1mm} W$.
The boundary of this manifold $T\times [0,1] \hspace{1mm} \#_b  \hspace{1mm} W$ is the disjoint union of $T\hspace{1mm}\#\hspace{1mm}\exo$ and $T$. 

We now build a new closed, smooth, orientable manifold $\mathcal{M}$ of dimension $4k$ as described below:
\begin{defn} $\mathcal{M}:=(T \times [0,1] \hspace{1mm} \#_b  \hspace{1mm} W)_\mathcal{F}$.
\end{defn}

Let $D$ be a $4k$ dimensional disk. The boundary of $D$ is $S^{4k-1}$. The boundary of $W$ is $\exo$ which is homeomorphic to $S^{4k-1}$. 
By gluing the boundaries of $W$ and $D$ via this homeomorphism we obtain a compact topological manifold without boundary, which we will 
denote as $W\cup D$.

\begin{prop}\label{prop1}
$\mathcal{M}$ is homeomorphic to $(T\times S^1)\hspace{1mm}\#\hspace{1mm}(W \cup D)$, where $D$ is a $4k$-dimensional disk.
\end{prop}

\begin{proof}
 Let us take the connect sum of $T\times[0,1]$ by removing a point from its interior with the topological manifold $W\cup D$ and denote
 the resulting manifold with boundary by $\overline{W}$. We write:
 \begin{displaymath}
  \overline{W}=(T\times [0,1])\hspace{1mm}\#\hspace{1mm}(W \cup D)
 \end{displaymath}
 Note that $\overline{W}$ has two boundary components, both being the manifold $T$ and that $\overline{W}$ is homeomorphic to  
 $(T\times [0,1])\hspace{1mm}\#_b\hspace{1mm}W$. Note that there is a canonical homeomorphism between $T\hspace{1mm}\#\hspace{1mm}\exo$ and $T$.
 
 Recall that by Proposition \ref{prop:pseu} we have $\mathcal{F}$ is topologically pseudo-isotopic to the identity map.
 
 Now by Lemma \ref{lem:lem2} and the fact that $\overline{W}$ is homeomorphic to  
 $(T\times [0,1])\hspace{1mm}\#_b\hspace{1mm}W$, we can conclude that $\mathcal{M}:=(T \times [0,1] \hspace{1mm} \#_b  \hspace{1mm} W)_\mathcal{F}$
 is homeomorphic to $(\overline{W})_{id}$. 
 
 Let us note that:
\begin{displaymath}
 (\overline{W})_{id} = (T\times S^1)\hspace{1mm}\#\hspace{1mm}(W \cup D)
\end{displaymath}
 where $id:T\rightarrow T$ is the identity map. \\
 
Therefore we have proved that $\mathcal{M}$ is homeomorphic to $(T\times S^1)\hspace{1mm}\#\hspace{1mm}(W \cup D)$.
 
\end{proof}

For any manifold $N$ let $\sigma(N)$ denote the signature of the manifold.

By the proposition we proved above, for our $4k$-dimensional manifold $\cM$,
\begin{displaymath}
 \sigma(\cM)=\sigma((T\times S^1)\hspace{1mm}\#\hspace{1mm}(W \cup D))
\end{displaymath}
since the signatures of two homeomorphic manifolds are equal. 
By the properties of the signature we get, 
\begin{displaymath}
 \sigma((T\times S^1)\hspace{1mm}\#\hspace{1mm}(W \cup D))=\sigma(T\times S^1)+\sigma(W \cup D)
\end{displaymath}
\begin{displaymath}
 =\sigma(T)\sigma(S^1)+\sigma(W \cup D)
\end{displaymath}
\begin{displaymath}
 =\sigma(W\cup D)
\end{displaymath}

Also by definition:
\begin{displaymath}
 \sigma(W\cup D)=\sigma(W).
\end{displaymath}

Therefore
\begin{displaymath}
 \sigma(\cM)=\sigma(W)
\end{displaymath}

For the sake of notational brevity let us assume the following:
\begin{enumerate}
 \item $X$ will denote $(T\times S^1)\hspace{1mm}\#\hspace{1mm}(W \cup D)$
 \item $A$ will denote $T\times S^1$
 \item $C$ will denote $W \cup D$
\end{enumerate}

We should note that $X$ and $C$ are not smooth manifolds, but since we will be using the fact that the smooth manifold
$\cM$ is homeomorphic to $X$ which in turn is homeomorphic to $A\#C$ we just consider the manifolds $X$ and $C$ as topological manifolds and 
it is enough for our purposes here. 

Now let us approach $\sigma(\cM)$ in terms of its Pontryagin numbers using the Hirzebruch Signature Theorem:
$$\sigma(\cM)=\cL[\cM]=\langle \cL_k(\cM),[\cM] \rangle $$

where $\cL[\cM]$ is the $\cL$-genus of the $4k$-dimensional manifold $\cM$.
$\cL_k(\cM)$ is a polynomial in the first $k$ Pontryagin classes of $\cM$ with rational coeffcients.

Our claim now, is that all other terms except the ``leading term'' of $\cL(\cM)$ which involves $p_k(\cM)\in H^{4k}(\cM)$ all other terms
pair up with the fundamental class $[\cM]$ to produce zero. In other words except $\langle p_k(\cM),[\cM] \rangle$ all other Pontryagin numbers
$$\langle p_{\alpha_1}^{m_1}p_{\alpha_2}^{m_2}\cdots p_{\alpha_r}^{m_r}(\cM),[\cM] \rangle=0$$ 

where $\alpha_1m_1+\alpha_2m_2+\cdots +\alpha_rm_r=k$.

So we work now to verify this claim:

Recall our notation of $X, A$ and $C$, also let $D_A$ and $D_C$ are disks (homeomorphic images of $D^{4k}$) in $A$ and $C$ respectively.

We see that by Mayer Vietoris sequence,

$$\cdots \rightarrow H^{i-1}(S^{4k-1}\times I)\rightarrow H^i(X) \rightarrow H^i(A\setminus D_A)\oplus H^i(C\setminus D_C)\rightarrow H^i(S^{4k-1}\times I)
\rightarrow \cdots$$

Since,

\begin{displaymath}
 H^{i-1}(S^{4k-1}\times I)=H^i(S^{4k-1}\times I)=0\,\,\,\text{for}\,\,\, 2<i<4k-1
\end{displaymath}
we have,
\begin{displaymath}
 H^i(X)\cong H^i(A \smallsetminus D_A)\oplus H^i(C \smallsetminus D_C)\,\,\,\text{for}\,\,\, 2<i<4k-1
\end{displaymath}
It is not difficult to see that 
\begin{displaymath}
 H^i(A \smallsetminus D_A)=H^i(A)\hspace{2mm}
 \hspace{2mm}\text{for}\hspace{2mm} 2<i<4k-1
\end{displaymath}
Therefore, we have
\begin{displaymath}
 H^i(\cM)\cong H^i(X)\cong H^i(A)\oplus H^i(C\setminus D_C)\,\,\,\text{for}\,\,\, 2<i<4k-1
\end{displaymath}

Note that all Pontryagin classes of $\cM$ except $p_k(\cM)$ and $p_0(\cM)=1$ are in some $H^i(\cM)$ where $3<i<4k-3$.
Also by Novikov's Theorem proving the topological invariance of rational Pontryagin classes the above isomorphism maps:
\begin{displaymath}
 p_j(\cM)\mapsto (p_j(A),p_j(C\setminus D_C))\,\,\,\text{for}\,\,\, 1\le j \le k-1
\end{displaymath}
Since $C=W\cup D$ and $D_C$ can be chosen to be $D$, the inclusion $W \hookrightarrow W \cup D=C$ gives us the first equality in the equation below, and since $W$
is parallelizable we get the second equality:
$$p_j(C\setminus D_C)=p_j(W)=0\,\,\,\text{for}\,\,\, 1\le j \le k-1$$ 

So, $$p_j(\cM)\mapsto (p_j(A),0)\,\,\,\text{for}\,\,\, 1\le j \le k-1$$

$$\therefore \langle p_{\alpha_1}^{m_1}p_{\alpha_2}^{m_2}\cdots p_{\alpha_r}^{m_r}(\cM),[\cM] \rangle
= \langle p_{\alpha_1}^{m_1}p_{\alpha_2}^{m_2}\cdots p_{\alpha_r}^{m_r}(A),[A] \rangle$$

provided $1\le \alpha_i \le k-1$ for all $i$ which is the case when $\alpha_1m_1+\alpha_2m_2+\cdots +\alpha_rm_r=k$.

And since, $A=M\times S^1 \times S^1=\partial(M\times D^2 \times S^1)$

$$ \langle p_{\alpha_1}^{m_1}p_{\alpha_2}^{m_2}\cdots p_{\alpha_r}^{m_r}(A),[A] \rangle=0$$ as we know that all Pontryagin numbers are zero for
manifolds that bound, which gives us
$$\langle p_{\alpha_1}^{m_1}p_{\alpha_2}^{m_2}\cdots p_{\alpha_r}^{m_r}(\cM),[\cM]\rangle=0$$ when $\alpha_1m_1+\alpha_2m_2+\cdots +\alpha_rm_r=k$,
excepting the case $r=1, \alpha_1=p_k$ and $m_1=1$.

So the only surviving Pontryagin class in the $\cL_k$-polynomial of $\cM$ is $p_k$. The coeffcient $s_k$ of $p_k$ in $\cL_k$ is given in page 12 of \cite{HIR} by:
\begin{displaymath}
 s_k=\displaystyle{\frac{2^{2k}(2^{2k-1}-1)}{(2k)!}B_k}
\end{displaymath}
where $B_k$ is the $k$-th Bernoulli number.

Summarizing the above work,
\begin{equation}\label{1}
\sigma(W)=\sigma(\cM)=\displaystyle{\displaystyle{\frac{2^{2k}(2^{2k-1}-1)}{(2k)!}}B_k \langle p_k(\cM),[\cM]\rangle}
\end{equation}

Let us now state a few things about $W$ and $\sigma(W)$, which are stated and proved in \cite{MIL-KER}.

Let $\Theta^n$ be the abelian group of oriented diffeomorphism class of $n$-dimensional homotopy spheres with the connected sum operation. Then, the oriented
diffeomorphism classes of $n$-dimensional homotopy spheres that bound a parallelizable manifold form a subgroup of $\Theta^n$ and will be denoted by 
$\Theta^n(\partial\pi)$.

The following theorems from \cite{MIL-KER} shall be used shortly:
\begin{thm}
 The group $\Theta^n(\partial\pi)$ is a finite cyclic group. Moreover if $n=4k-1$, then 
 \begin{displaymath}
  \Theta^n(\partial\pi)=\mZ_{t_k} 
 \end{displaymath}
where $t_k=a_k2^{2k-2}(2^{2k-1}-1)num\displaystyle{\left(\frac{B_k}{4k}\right)}$ and $a_k=1$ or $2$ if $k$ is even or odd.
\end{thm}
In the above theorem and in what follows,
$num\displaystyle{\left(\frac{a}{b}\right)}$ stands for the numerator of the fraction $\displaystyle\frac{a}{b}$ in its lowest
term.

\begin{thm}\label{thm:milker}
 Let $t_k$ be the order of $\Theta^{4k-1}(\partial\pi)$. There exists a complete set of representatives 
 $\displaystyle\{S^{4k-1}=\Sigma_0,\Sigma_1, \dots, \Sigma_{t_k-1}\}$ in $\Theta^{4k-1}(\partial\pi)$ such that if $W$ is a parallelizable manifold with 
 boundary $\Sigma_i$ then $\displaystyle{\left(\frac{\sigma(W)}{8}\right)mod\, t_k=i}$.
\end{thm}

Therfore if we choose some $\Sigma_i$ from the set of representatives in the above theorem, with $\partial W=\Sigma_i$, then
\begin{equation}\label{2}
 \sigma(W)=8(t_kd+i) \,\,\,\,\, \text{for some postive integer}\,\,\,d.
\end{equation}

Rewriting equation (\ref{1}) using the equation (\ref{2}) we get
\begin{equation}
 8(t_kd+i)=\displaystyle{\displaystyle{\frac{2^{2k}(2^{2k-1}-1)}{(2k)!}}B_k \langle p_k(\cM),[\cM]\rangle}.
\end{equation}

Now we state another result from chapter 2 of Ardanza's thesis \cite{ARD}:
\begin{thm}
 For all $k > 1$, there exists a prime $p > 2k+1$ such that $p$ divides num$\displaystyle{\left(\frac{2^{2k}(2^{2k-1}-1)}{(2k)!}B_k\right)}$.
\end{thm}

\begin{prop}\label{prop2}
 If prime $p > 2k+1$, $k>1$, is such that $p$ divides num$\displaystyle{\left(\frac{2^{2k}(2^{2k-1}-1)}{(2k)!}B_k\right)}$ then $p$ divides $t_k$.
\end{prop}

\begin{proof}
 Let $(B_k)$ in lowest terms be $\frac{X}{Y}$.
 
 Let the exponents of $p$ in the unique prime factor decompositions of $2^{2k}(2^{2k-1}-1)$, $X$, $4kY$ and $(2k)!Y$ be $m, n, r$ and $s$ respectively.
 
 We know that, $p$ divides $num\displaystyle{\left(\frac{2^{2k}(2^{2k-1}-1)}{(2k)!}B_k\right)}$.

 Therefore, $p$ divides $num\displaystyle{\left(\frac{2^{2k}(2^{2k-1}-1)X}{(2k)!Y}\right)}$. 
 
 From this we conclude, $m+n>s$.
 
 But since $4kY$ divides $(2k)!Y$, we conclude $s>r$. Putting the two inequalities together we get, $m+n>r$. Therefore, 
 $p$ divides $2^{2k-2}(2^{2k-1}-1)num\displaystyle{\left(\frac{X}{4kY}\right)}$.
 
 Hence, $p$ divides $2^{2k-2}(2^{2k-1}-1)num\displaystyle{\left(\frac{B_k}{4k}\right)}$.
 
 This is sufficient to conclude $p$ divides $t_k$.
\end{proof}

In equation (3), $\langle  p_k,[\cM]\rangle$ is an integer since $\cM$ is
smooth, and the prime $p$ divides $t_k$ and $\displaystyle{\frac{2^{2k}(2^{2k-1}-1)}{(2k)!}}B_k$
so it has to divide $i$. But we have the freedom of choosing $\Sigma_i$ from the set of representatives in Theorem \ref{thm:milker}
such that $i$ is not divisible by $p$. We could always pick $i=1$ and that will work.

Summarizing the work done so far, we can always choose a $\Sigma_i$ and use the methods of section $2$ to get a self-diffeomorphism 
$f:M\rightarrow M$ such that $[f]\in \pi_0(\text{Diff}_0(M))$ and $[f] \ne 0$ in $\pi_0(\text{Diff}(M))$.

\section{Showing that [$f$] maps to 0 in $\pi_0(\Met^{<0}(M))$}\label{subsec:met}
The proof of Theorem $2$ of \cite{FO1} given on pages 53-54 of that paper can be seen to yield the following result whose proof (for the reader's convenience)
we sketch at the end of this section.

\begin{prop}
Fix an exotic sphere $\exo^{4k-1}$ representing an element in $\Theta_{4k-1} (k\ge 2)$ and let 
$h:S^{4k-3}\times [1,2]$ be the diffeomorphism construced from it in section $2$. Then there exists a real number $\alpha > 0$ depending only on $h$
such that the following is true. Let $(M^{4k-2}, g)$ be a closed negatively curved Riemannian manifold which contains a codimension-$0$ ball $B$ 
isometric to a ball of radius $3\alpha$ in (real) hyperbolic n-space $\mathbb{H}^n, n = 4k-2$. And let $f:M^n\rightarrow M^n$ be the self diffeomorphism
constructed using $h$ and $\alpha$ by the method given in section $2$. Then $f_*g$ and $g$ lie in the same path component of Met$^{<0}(M^n)$.
\end{prop}

One also easily concludes the following result by the discussion given in section $0$ of \cite{FJ3}, page 58.

\begin{prop}
Given a closed complex hyperbolic manifold $(N^n, g_\mathbb{C})$ of real dimension $n$ and a real number $r\ge1$, there exists
a finite sheeted cover $\widetilde{N^n}$ of $N^n$ and a (special) negatively curved Riemannian metric $g_s$ on $\widetilde{N^n}$ such that
$(\widetilde{N^n}, g_s)$ contains an embedded codimension-$0$ ball which is isometric to a ball of radius $3r$ in $\mathbb{H}^n$ and outside of a
concentric embedded geodesic ball of radius $9r^2$ the Riemannian metrics $g_s$ and $\widetilde{g_\mathbb{C}}$ coincide. Here $g_\mathbb{C}$ denotes
the complex hyperbolic metric on $N^n$ and $\widetilde{g_\mathbb{C}}$ is the induced complex hyperbolic metric on the covering space $\widetilde{N^n}$.
\end{prop}

We now construct the manifold $M^n (n=4k-2, k\ge 2)$ posited to exist in Theorem $1.1$. Let $\exo^{4k-1}$ be the exotic sphere $\exo$ of Theorem $3.11$ and 
$N^{4k-2}$ be any closed complex hyperbolic manifold of real dimension $n=4k-2$. Let $\alpha > 0$ be the real number determined in Proposition 4.1 by $\exo^{4k-1}$.
Then $M^n$ is the finite sheeted cover $\widetilde{N^n}$ of $N^n$ posited to exist in Proposition 4.2 using $N^n$ and setting $r=\alpha$. And Proposition 4.2
also furnishes $M^n$ with a special negatively curved Riemannian metric $g_s$ to which Proposition 4.1 can be applied. 

Now if we choose $g_s$ as the basepoint of Met$^{<0}(M)$ and look at the exact sequence in section $1$ (under idea of proof), 
since $f_*g_s$ and $g_s$ are in the 
same path component of Met$^{<0}(M)$, the isotropy class $[f]$ of $f$ pulls back to an element $a\in \pi_1(\cT^{<0}(M),[g])$. And this element $a$ cannot
be the identity element since $[f]\ne[id_M]$, by Theorem $3.12$ and the discussion following it.

\begin{rem}
We do not know whether or not the equivalence classes of the
Riemannian metrics $\tilde{g}_\mathbb{C}$ and $g_s$ from Proposition 4.2 lie in the
same arc component of $\cT^{<0}(\widetilde{N^n})$. We hope to settle this
question in a future paper.
\end{rem}

{\emph{Proof of Proposition 4.1.}}

Denote by $B\subset M$ the closed geodesic ball centered at $p$ of radius $3\alpha$. 
Let us also denote the metric $g$ on $M$ restricted to $B$ as $g^0$, where $g^0$ is the hyperbolic metric. 
As before we identify $B\smallsetminus{p}$ with $S^{4k-3}\times(0,3\alpha]$. This idenification can be done isometrically: $B\smallsetminus{p}$ with metric
$g^0$ is isometric to $S^{4k-3}\times(0,3\alpha]$ with metric $\sinh^2(t)\overline{h}+dt^2$, where $\overline{h}$ is the Riemannian metric on the 
sphere $S^{4k-3}$ with constant curvature equal to $1$. In view of this identification, we write
\begin{displaymath}
 g^0(x,t)=\sinh^2(t)\overline{h}+dt^2.
\end{displaymath}Here we

Recall that,\newline
\centerline {$\displaystyle f(q)  = \begin{cases} 
      q & \text{if } q \notin (S^{4k-3}\times[\alpha,2\alpha])\subset M \\
      (h_{t/\alpha}(x),t) & \text{if } q=(x,t) \in (S^{4k-3}\times[\alpha,2\alpha]) \subset M \\
     \end{cases}$}\\\\
where $h(x,s)=(h_s(x),s)$.\newline
Also recall that, $h:S^{4k-3}\times [1,2]\rightarrow S^{4k-3}\times [1,2]$ is level preserving, and identity near $1$ and $2$.

Therefore, the metric $g^1=f_*g^0$ (the push forward of $g^0$ by $f$) on $B\smallsetminus {p}$ is given by:
\begin{displaymath}
  {\displaystyle g^1(x,t)  = \begin{cases} 
      g^0(x,t) & \text{if } t \notin [\alpha,2\alpha] \\
      h_*g^0(x,t) & \text{if } t \in [\alpha,2\alpha]. \\
     \end{cases}}\\\
\end{displaymath}

Let us state a lemma from \cite{FO1}.
\begin{lemma}
 Let $\mathscr G' \subset \text{Diff}_0(S^{n-1}\times [1,2])$ be the group of all smooth isotopies $h$ of the $(n-1)$-dimensional sphere $S^{n-1}$ that are the 
 identity near $1$ and constant near $2$. Then $\mathscr G'$ is contractible.
\end{lemma}

The diffeomorphism $h:S^{4k-3}\rightarrow S^{4k-3}$ that we use for the contruction of our $f$ is an element of $\mathscr G'$.
Therefore by the above Lemma, there is a path of isotopies $h^\mu \in \mathscr G'$ for $\mu\in [0,1]$, with $h^0=h$ and $h^1= id_{S^{4k-3}\times [1,2]}$. 
Each $h^\mu$ is a smooth isotopy of the sphere $S^{4k-3}$. Let us denote the final map in the isotopy $h^\mu$ as $\theta^\mu$, that is, 
\begin{displaymath}
 h^\mu(x,2)=(\theta^\mu(x),2)
\end{displaymath}
Note that $\theta^\mu:S^{4k-3}\rightarrow S^{4k-3}$ is a diffeomorphism and $\theta^0=\theta^1=id_{S^{4k-3}}$.
Define
\begin{displaymath}
 \phi^\mu:S^{4k-3}\times[\alpha,2\alpha]\rightarrow S^{4k-3}\times[\alpha,2\alpha]
\end{displaymath}
by rescaling $h$ to the interval $[\alpha,2\alpha]$, that is, $\phi^\mu(x,t)=(h^\mu_{t/\alpha}(x),t)$.
Let $\delta:[2,3]\rightarrow [0,1]$ be smooth with $\delta(2)=1$,$\delta(3)=0$, and $\delta$ constant near $2$ and $3$. Now we are in a postion to define
a path of negatively curved metrics $g^\mu$ on $B\smallsetminus {p}=S^{4k-3}\times(0,3\alpha]$:

\begin{displaymath}
  {\displaystyle g^\mu(x,t)  = \begin{cases} 
      g^0(x,t) & \text{if } t \in (0,\alpha] \\
      (\phi^\mu)_*g^0(x,t) & \text{if } t \in [\alpha,2\alpha] \\
      \sinh^2(t)(\delta(\frac{t}{\alpha})(\theta^\mu)_*\overline{h}(x)+(1-\delta(\frac{t}{\alpha}))\overline{h}(x)) + dt^2 & \text{if } t\in [2\alpha,3\alpha]\\
     \end{cases}}\\\
\end{displaymath}

Since $\delta$ and all isotopies used are constant near the endpoints of the intervals on which they are defined, it is clear that $g^\mu$ is a smooth metric
on $B\smallsetminus {p}$ and that $g^\mu$ joins $g^1$ to $g^0$. Moreover $g^\mu(x,t)=g^0(x,t)$ for $t$ near $0$ and $3$. Hence we can extend $g^\mu$ to the whole 
manifold $M$ by defining $g^\mu(q)=g^0(q)$ for $q=p$ and $g^\mu(q)=g_s(q)$ when $q\notin B$.

The metric $g^\mu(x,t)$ is equal to $g^0(x,t)$ for $t\in(0,\alpha]$; hence $g^\mu(x,t)$ is hyperbolic for $t \in (0,\alpha]$. Also, $g^\mu(x,t)$ 
is the push-forward (by $\phi^\mu$) of the hyperbolic metric $g^0$ for $t\in [\alpha, 2\alpha]$; hence $g^\mu(x,t)$ is hyperbolic for $t\in[\alpha,2\alpha]$.
For $t\in[2\alpha,3\alpha]$, the metric $g^\mu(x,t)$ is similar to the ones constructed in \cite{FJ2} or in Theorem 3.1 in \cite{O}.

It can be checked from those references that the sectional curvatures of $g^\mu$ are $\epsilon$ close to $-1$, provided $r$ is large enough. 

Therefore, $f_*g$ and $g$ lie in the same path component of Met$^{<0}(M)$.
\bigskip

%%%%%%%%%%%%%%%%%%%%%%%%%%%%%%%%

\end{document}